\documentclass[12pt]{amsart}

\usepackage{verbatim}
\usepackage{tikz}
\usetikzlibrary{cd,matrix,arrows,positioning,calc} 
\tikzset{
    >=stealth',
    punkt/.style={
           rectangle,
           rounded corners,
           draw=black, very thick,
           text width=6.5em,
           minimum height=2em,
           text centered},
    pil/.style={
           ->,
           thick,
           shorten <=2pt,
           shorten >=2pt,}
}
\usepackage[ top=.2in,
  bottom=1in,
  left=0.8in,
  right=0.8in,, headheight=15pt, includehead]
{geometry}                % See geometry.pdf to learn the layout options. There are lots.
\usepackage{fancyhdr}
\usepackage{wrapfig}
\usepackage[font=footnotesize, labelfont=bf]{caption}
\usepackage{graphicx}
\usepackage[english]{babel}
\usepackage[applemac]{inputenc}
\usepackage{epstopdf}
\usepackage{amssymb, amsthm, amsfonts, amsmath, mathtools}
\usepackage[all]{xy}
\usepackage{hyperref}
\newtheorem{theorem}{Theorem}
\newtheorem{proposition}[theorem]{Proposition}

\newtheorem{corollary}[theorem]{Corollary}
\newtheorem{lemma}[theorem]{Lemma}
\newtheorem*{theorem*}{Theorem}
\newtheorem*{proposition*}{Proposition}
\newtheorem*{question*}{Question}
\newtheorem*{conjecture*}{Conjecture}

\theoremstyle{definition}

\title{A\MakeLowercase{bout the even minimal stratum of translation surfaces in genus $4$}}

\author{R\MakeLowercase{iccardo} G\MakeLowercase{iannini}}

\begin{document}
\maketitle

\begin{abstract}
\noindent 
   \footnotesize In the present note, we complete the correspondence between stratum components of translation surfaces in low genus and finite-type Artin groups with defining Dynkin diagram containing $E_6$. In an earlier work, we showed that in genus $3$ the monodromy of the non-hyperelliptic connected components $\mathcal{H}^{\operatorname{odd}}(4)$ and $\mathcal{H}(3,1)$ are highly non-injective, as the respective kernels contain a non-abelian free group of rank $2$. The result holds since both the stratum components are orbifold classifying spaces for central extensions of the inner automorphism groups of the finite-type Artin groups $A_{E_6}$ and $A_{E_7}$, respectively. The following is a note extending the same result to the stratum $\mathcal{H}^{\operatorname{even}}(6)$ in genus $4$, which is an orbifold classifying space for a central extension of the group $\operatorname{Inn}(A_{E_8})$.
\end{abstract}

\section*{Introduction}

We study the topology and the topological monodromy of a non-hyperelliptic stratum component of translation surfaces in genus $4$, namely the minimal stratum denoted by $\mathcal{H}^{\operatorname{even}}(6)$. Strata of translation surfaces are moduli spaces parametrizing families of \textit{translation surfaces}, obtained by identifying the sides of polygons in the complex plane through translations. After the identification through translations, the vertices of the defining polygons come with an angle of $2(k_i+1)\pi$ for some positive integers $k_i$. The data these integers give stratify the whole moduli space of translation surfaces in spaces denoted by $\mathcal{H}(k_1,\dots,k_n)$. The strata are not necessarily connected.\\

Kontsevich-Zorich described the connected components of the strata \cite{Kontsevich2003} and proved that some stratum components are hyperelliptic, parameterizing translation surfaces that are branched double covers of spheres. The hyperelliptic stratum components are orbifold classifying spaces for some finite extension of braid groups \cite[Section 2]{Looijenga2014}. Looijenga-Mondello proved that the non-hyperellipitic stratum components $\mathcal{H}^{\operatorname{odd}}(4)$ and $\mathcal{H}(3,1)$ are orbifold classifying spaces for a central extension of the inner automorphism group of some \textit{Artin groups} $A_\Gamma$ \cite[Theorem 1.1]{Looijenga2014}, finitely presented groups with a presentation given by a finite tree $\Gamma$. We extend the result to  the stratum component $\mathcal{H}^{\operatorname{even}}(6)$.

\begin{theorem}
The stratum components $\mathcal{H}^{\operatorname{even}}(6)$ is an orbifold classifying space for a central extension of $\operatorname{Inn}(A_{E_8})$, the inner automorphism group of the Artin group of type $E_8$.
\end{theorem}

In an earlier work \cite{Giannini2023}, we studied the kernel of the topological monodromy maps of $\mathcal{H}^{\operatorname{odd}}(4)$ and $\mathcal{H}(3,1)$, homomorphisms that compare the orbifold fundamental groups $\pi_1^{orb}(\mathcal{H}^{\operatorname{odd}}(4))$ and $\pi_1^{orb}(\mathcal{H}(3,1))$ to the associated mapping class groups. Calderon-Salter described the images \cite{CalderonSalterFramed2022}, and it is natural to ask if monodromies are injective: kernel elements can be interpreted as homotopy classes of loops of some covering topological spaces, called \textit{Teichm\"uller strata of translations surfaces}. We proved that the monodromy maps of the stratum components $\mathcal{H}^{\operatorname{odd}}(4)$ and $\mathcal{H}(3,1)$ are highly non-injective, suggesting that Teichm\"uller strata might have an intricate topology. In particular, we proved that the kernels of the monodromies associated with the stratum components $\mathcal{H}(3,1)$ and $\mathcal{H}^{\operatorname{odd}}(4)$ contain a non-abelian free group of rank $2$. We extend our work to an additional stratum component.

%and Kontsevich-Zorich conjectured that the other \textit{non-hyperellipitc stratum components} should also be covered by contractible spaces \cite[Section 4]{Kontsevich1997}.  Moreover, Kontsevich-Zorich also conjectured that the orbifold fundamental groups are commensurable with some subgroups of \textit{mapping class groups}, the groups of self-homeomorphisms of finite-type surfaces modulo homotopy. Few results are known about the topology of non-hyperelliptic components.\\

%Looijenga-Mondello \cite[Theorem 1.1]{Looijenga2014} proved that the non-hyperellipitic stratum components $\mathcal{H}^{\operatorname{odd}}(4)$ and $\mathcal{H}(3,1)$ are orbifold classifying spaces for a central extension of the inner automorphism group of some \textit{Artin groups} $A_\Gamma$, finitely presented groups with a presentation given by a finite tree $\Gamma$. 
%$\Gamma$ as follows: if $\mathcal{V}(\Gamma)=\{v_1,\dots,v_n\}$ is the set of vertices of $\Gamma$, then the Artin group of type $\Gamma$ is  
%\begin{align*}
%    A_\Gamma = \Bigg\langle a_1, \dots, a_n  \text{ }\Bigg\vert
%    \text{ }
%    \begin{array}{@{}l@{}l@{}l@{}}
%         a_i a_j a_i &= a_j a_i a_j &\quad \text{if } v_i \text{ and } v_j \text{ are adjacent}\\
 %        a_i a_j &= a_j a_i &\quad \text{otherwise}
  %  \end{array}
   % \Bigg\rangle. 
%\end{align*}  

\begin{theorem}
The kernel of the monodromy $\rho:\pi_1^{orb}(\mathcal{H}^{\operatorname{even}}(6))\rightarrow\operatorname{Mod}(\Sigma_{4,1})$ contains a non-abelian free group of rank $2$.
\end{theorem}

\subsection*{\small Structure of the paper.} Section $1$ covers some preliminary notions about the strata of translation surfaces. In Section $2$, we use gap sequences to describe the Riemann surfaces in the stratum component $\mathcal{H}^{\operatorname{even}}(6)$. In Section $3$, we introduce versal deformation spaces of plane curve singularities and Artin groups to study $\mathcal{H}^{\operatorname{even}}(6)$. In the last two section, we describe the orbifold fundamental group and prove that the kernel of the monodromy contains a non-abelian free group of rank $2$.

\subsection*{\small Acknowledgments.} The author would like to thank Tara Brendle and Vaibhav Gadre for their support during the drafting of the present note, as well as Dawei Chen for the inspiring conversation in Oberwolfach during the workshop ``Riemann Surfaces: Random, Flat, and Hyperbolic Geometry", and both Philipp Bader and Franco Rota for useful discussions on an earlier draft. The author is
supported by EPSRC grant EP/T517896/1.

\section{Preliminaries}

Strata of translation surfaces parameterize pairs $(X,\omega)$ where $X$ is a closed genus $g$ Riemann surface and $\omega$ is an abelian differential on $X$ with prescribed orders at the vanishing points\footnote{This definition of translation surface is equivalent to the one given in the introduction \cite[Section 1]{Wright2015}. Indeed, any holomorphic section of the cotangent bundle $\omega$ defined a polygonal representation, and viceversa.}. More precisely, let $k_1,\dots,k_n$ and $g$ be positive integers such that $\sum_{i=1}^nk_i=2g-2$. The stratum $\mathcal{H}(k_1,\dots,k_n)$ is the set of all equivalence classes of genus $g$ translation surfaces $(X,\omega)$ such that that $\omega$ has vanishing locus consisting of $n$ points $p_1,\dots,p_n$ such that $\operatorname{ord}_{p_i}(\omega)=k_i$ for every $i\in\{1,\dots,n\}$. In particular, any two translation surfaces $(X_1,\omega_1)$ and $(X_2,\omega_2)$ are in the same equivalence class of $\mathcal{H}(k_1,\dots,k_n)$ if and only if there is some bilohomorphism $I:X_1\rightarrow X_2$ such that $I^*\omega_2=\omega_1$.\\

Strata are not necessarily connected, and Kontsevich-Zorich classified the connected components \cite{Kontsevich2003}. The minimal strata $\mathcal{H}(2g-2)$ in genus $g\geq 4$ have exactly three connected components. One stratum component parametrizes hyperelliptic translation surfaces $(X,\omega)$ where $X$ is a hyperelliptic Riemann surface and the hyperelliptic involution $\iota\in\operatorname{Aut}(X)$ preserves the abelian differential $\omega\in\Omega^1(X)$ up to sign, that is $\iota^*\omega=-\omega$. The remaining two stratum components do not contain hyperelliptic Riemann surfaces and are \textit{totally non-hyperelliptic}. The non-hyperelliptic components on $\mathcal{H}(2g-2)$ are distinguished by the parity of their \textit{spin structures}. \\

Indeed, a holomorphic abelian differential $\omega\in\Omega^1(X)$ on a closed Riemann surface $X$ defines an effective canonical divisor $\operatorname{div}(\omega)$. If $(X,\omega)$ is a translation surface in a stratum of the form $\mathcal{H}(2k_1,\dots,2k_n)$, the divisor $\operatorname{div}(\omega)/2$ corresponds to a section of some line bundle $\mathcal{L}$ on $X$ such that $\mathcal{L}^{\otimes 2}$ is the canonical cotangent bundle $K_X$ of $X$. The line bundle $\mathcal{L}$ is a spin structure of $(X,\omega)$ and its parity is the complex dimension $h^0(X,\mathcal{L})\mod 2$ of the space of holomorphic sections $ X\rightarrow \mathcal{L}$. The non-hyperelliptic components of the minimal strata $\mathcal{H}(2g-2)$ are denoted by $\mathcal{H}^{\operatorname{odd}}(2g-2)$ and $\mathcal{H}^{\operatorname{even}}(2g-2)$ according to the parity of the spin structures of their abelian differentials.\\

Strata of translation surfaces come with a natural topology inherited by their \textit{Teichm\"uller covers}, similarly as the moduli spaces of genus $g$ Riemann surfaces $\mathcal{M}_g$ get their topology from the Teichm\"uller spaces $\mathcal{T}_g$. To define the Teichm\"uller cover $\mathcal{TC}$ of a stratum component $\mathcal{C}$, we begin by fixing a topological genus $g$ surface with $n$ marked points $q_1,\dots,q_n$ denoted by $\Sigma_{g,n}$. The space $\mathcal{TC}$ is the set of triples $(X,f,\omega)$ where $(X,\omega)\in\mathcal{C}$ and $f$ is a \textit{marking} for $(X,\omega)$. A marking of a translation surface $(X,\omega)$ is the homotopy class of a diffeomorphism $\Sigma_{g,n}\rightarrow X\setminus$ such that $f(q_i)=p_i$ and the homotopies are considered relative to the set of marked points. In other words, two triples $(X_1,f_1,\omega_1)$ and $(X_2,f_2,\omega_2)$ are in the same equivalence class if and only if there exists a bilohomorphism $I:X_1\rightarrow X_2$ such that $I^*\omega_2=\omega_1$ and $f_2$ is homotopic to $I\circ f_1$ relative the marked points.\\

The Teichm\"uller cover of a stratum component $\mathcal{C}$ of $\mathcal{H}(k_1,\dots,k_n)$ is a manifold of dimension $2g+n-1$. In order to define its topology, we begin by fixing a triangulation $\tau$ of $\Sigma_{g,n}$ where the vertices are marked points. We also define the set $U_\tau$ of triples $(X,f,\omega)\in\mathcal{T}\mathcal{C}$ where $f(\tau)$ is a triangulation of $(X,\omega)$ via saddle connections, namely geodesic arcs intersecting the zeros of $\omega$ only at the endpoints.  If $\{\gamma_1,\dots,\gamma_{2g+n-1}\}$ is a fixed basis of homology group $H_1(\Sigma_{g,n},\mathbb{Z})$, the map $U_\tau\rightarrow H^1(\Sigma_{g,n},\mathbb{C})$ assigning every triple $(X,f,\omega)$ the linear map $(\gamma_i\mapsto\int_{f_*\gamma_i}\omega)_{i=1}^{2g-n+1}$
is one-to-one and provide $\mathcal{T}\mathcal{C}$ with a smooth atlas.\\

We are interested in the isomorphism classes of the orbifold fundamental groups $\pi_1^{orb}(\mathcal{C})$, where $\mathcal{C}$ is any non-hyperelliptic stratum component. If $(X,f,\omega)$ is a fixed point in $\mathcal{TC}$, the group $\pi_1^{orb}(\mathcal{C})$
is the set of pairs $(\gamma,\phi)$ where $\phi\in\operatorname{Mod}(\Sigma_{g,n})$ $\gamma$ is the homotopy class, relative to the endpoints, of an arc in $\mathcal{TC}$ connecting $(X,f,\omega)$ with $(X,f\phi^{-1},\omega)$. The binary operation on $\pi_1^{orb}(\mathcal{C})$ is given by the composition law $(\gamma_1,\phi_1)(\gamma_2,\phi_2)=(\gamma_1*(\phi_1\cdot\gamma_2),\phi_1\phi_2)$. In general, no description of the isomorphism classes of $\pi_1^{orb}(\mathcal{C})$ is available. However, in a few cases, we can compute a presentation for a quotient of $\pi_1^{orb}(\mathcal{C})$ that does not factor through the mapping class group $\operatorname{Mod}(\Sigma_{g,n})$ of the underlying punctured surface.\\

The punctured complex plane $\mathbb{C}^*$ acts continuously on any strata $\mathcal{H}(k_1,\dots,k_n)$ by multiplication on the abelian differentials.  The resulting quotient space, denoted by $\mathbb{P}\mathcal{H}(k_1,\dots,k_n)$, is a \textit{projective stratum}. If $\mathcal{C}$ is a connected component of $\mathcal{H}(k_1,\dots,k_n)$, the isomorphism classes of the orbifold fundamental group $\mathbb{P}\mathcal{C}$ has been computed by Looijenga-Mondello in most cases for $g=3$ \cite[Theorem 1.1, 5.3, 5.4]{Looijenga2014}.\\

Any projective stratum component is a \textit{good orbifold}: if $\mathcal{C}$ is a stratum component, its projectivization $\mathbb{P}\mathcal{C}$ is the quotient of a smooth manifold by the action of a discrete group. Specifically, there is a subgroup of the mapping class group $\operatorname{Mod}(\Sigma_{g,n})$ acting on the smooth manifold $\mathbb{P}(\mathcal{TC})$, defined as $\mathcal{TC}/\mathbb{C}^*$, so that the resulting quotient is $\mathbb{P}\mathcal{C}$. The punctured complex plane $\mathbb{C}^*$ acts freely and properly on each Teichm\"uller stratum component $\mathcal{TC}$ and by the Quotient Manifold Theorem $\mathbb{P}(\mathcal{TC})$ is a smooth manifold of dimension $\operatorname{dim}\mathcal{C}-1$. Moreover, the quotient map $q_\mathcal{C}:\mathcal{TC}\rightarrow\mathbb{P}(\mathcal{TC})$ is a smooth submersion and every $q_\mathcal{C}$ is a principal $\mathbb{C}^*$-bundle.\\

The relation between the orbifold fundamental groups of a stratum component $\mathcal{C}$ and its projectivization $\mathbb{P}\mathcal{C}$ is summarized below in Proposition \ref{projpi}. 

\begin{proposition}
\label{projpi}
    Let $\mathcal{C}$ be a stratum component. The following is a central short exact sequence
    $$0\rightarrow\pi_1(\mathbb{C^*})\rightarrow\pi_1^{orb}(\mathcal{C})\rightarrow\pi_1^{orb}(\mathbb{P}\mathcal{C})\rightarrow1.$$
\end{proposition}
\begin{proof}
The map $q_\mathcal{C}$ induces a surjection on fundamental groups since $\mathbb{C}^*$ is connected. Moreover, the same subgroup of the mapping class group acts on both $\mathbb{P}(\mathcal{TC})$ and $\mathcal{TC}$ so to obtain the quotients $\mathbb{P}\mathcal{C}$ and $\mathcal{C}$, respectively. The map $q_\mathcal{C}$ induces a surjection $\pi_1^{orb}(\mathcal{C})\twoheadrightarrow\pi_1^{orb}(\mathbb{P}\mathcal{C})$ on the orbifolds fundamental groups given by $(\gamma,\phi)\mapsto(q_{\mathcal{C}}(\gamma),\phi)$. The kernel is isomorphic to $\pi_1(\mathbb{C^*})$ and generated by a loop in a fiber of $q_\mathcal{C}$ that commutes with every pair $(\gamma,\phi)\in\pi_1^{orb}(\mathcal{C})$. 
\end{proof}

The projective stratum components $\mathbb{P}\mathcal{C}$ parameterize the isomorphism classes of pairs $(X,D)$, where $X$ is a closed Riemann surface and $D$ is an effective canonical divisor with prescribed multiplicities provided by the stratum component $\mathcal{C}$. The pairs $(X_1,D_1)$ and $(X_2,D_2)$ are equivalent in $\mathbb{P}\mathcal{C}$ if there exists a bilohomorphism $I:X_1\rightarrow X_2$ such that $I^*D_2=D_1$. If $\mathcal{C}$ is a minimal stratum component, its projectivization $\mathbb{P}\mathcal{C}$ can be projected in $\mathcal{M}_{g,1}$, the moduli space of pointed Riemann surfaces. 

\begin{proposition}
\label{iso}
    Let $\mathcal{C}$ be a connected component of  $\mathbb{P}\mathcal{H}(2g-2)$. The forgetful map $\mathbb{P}\mathcal{C}\rightarrow\mathcal{M}_{g,1}$ defined by $(X,(2g-2)p)\mapsto(X,p)$ is an orbifold isomorphism onto its image provided the dimension of the image is $2g-1$.
\end{proposition}
\begin{proof}
    The forgetful map is induced by the $\operatorname{Mod}(\Sigma_{g,1})$-equivariant continuous map $\mathbb{P}(\mathcal{T}\mathcal{C})\rightarrow\mathcal{T}_{g,1}$ given by mapping the triple $(X,f,(2g-2)p)$ to $(X,f,p)$. Any bijective continuous map between manifolds of the same dimension and without boundary is a homeomorphism by the invariance of domain Theorem.
\end{proof}

\section{Gap sequences in genus 4}

In this section, we describe the image of the minimal stratum component $\mathcal{H}^{\operatorname{even}}(6)$ in $\mathcal{M}_{4,1}$. A general reference is \cite[Chapter VII, Section 4]{Miranda}. \\

Each genus $g$  pointed closed Riemann surface $(X,p)$ comes with a sequence of $g$ integers $\mathfrak{G}_p(X)$ called the \textit{Weierstrass gap sequence}. A positive integer $n$ is a Weierstrass gap number in $\mathfrak{G}_p(X)$ if and only if there is an abelian differential on $X$ with a zero at $p$ of order $n-1$. The complement $\Gamma_p(X)$ of a gap sequence $\mathfrak{G}_p(X)$ in $\mathbb{N}$ is called a \textit{non-gap sequence} and it is a numerical semigroup. Given an arbitrary numerical semigroup $\Gamma$ in $\mathbb{N}$, we denote by $\mathcal{M}_{g,1}^{\Gamma}$ the moduli space of pointed Riemann surfaces $(X,p)$ such that the non-gap sequence at $p\in X$ is exactly $\Gamma$. If $X$ is hyperelliptic and $p$ is preserved by the hyperelliptic involution of $X$, the Weiestrass gap sequence $\mathfrak{G}_p(X)$ is $\{1,3,5,\dots,2g-1\}$. We will show the following.

\begin{proposition}
\label{gapstrata}
    Let $\Gamma$ be the semigroup generated by $3$ and $5$. Then, a pointed Riemann surface $(X,p)$ is in $\mathcal{M}_{4,1}^{\Gamma}$ if and only if $(X,6p)\in\mathbb{P}\mathcal{H}^{\operatorname{even}}(6)$.
\end{proposition}

Suppose $X$ is a non-hyperelliptic Riemann surface of genus $4$. Then, the class of the canonical divisors $K_X$ induces a holomorphic embedding $\phi:X\rightarrow\mathbb{P}^3$ of $X$ as a smooth degree $6$ curve. A consequence of Max Noether's Theorem for algebraic surfaces is that $X$ is the complete intersection of an irreducible quartic $Q$ and an irreducible cubic $C$ in $\mathbb{P}^3$. Irreducible quartics on $\mathbb{P}^3$ can either be smooth or singular cones. In the first case, the Segre embedding can be use to show that $Q$ is isomorphic to $\mathbb{P}^1\times\mathbb{P}^1$. Otherwise, the irreducible quadric $Q$ is a cone and, up to some change of coordinates, the vanishing locus of the homogeneous polynomial $x_0^2-x_1x_2$ in $\mathbb{P}^3$. The following can be found in \cite[Section 4.3]{Bullock}.

\begin{lemma}
\label{p3}
    Let $(X,6p)\in\mathbb{P}\mathcal{H}(6)$ and suppose $X$ is a non-hyperelliptic smooth degree $6$ curve in $\mathbb{P}^3$ that is a complete intersection of an irreducible quartic $Q$ and an irreducible cubic $C$. Then,
    \begin{itemize}
        \item if $Q$ is smooth, the Weierstrass gap sequence of $(X,p)$ is $\mathfrak{G}_p(X)=\{1,2,3,7\}$;
        \item if $Q$ is a cone, the Weierstrass gap sequence of $(X,p)$ is $\mathfrak{G}_p(X)=\{1,2,4,7\}$.
    \end{itemize}
\end{lemma}

\begin{proof}[Proof of Proposition \ref{gapstrata}]
 Suppose $(X,6p)\in\mathbb{P}\mathcal{H}^{\operatorname{even}}(6)$. By Lemma \ref{p3}, there are only two possible Weierstrass gap sequences in $p$. However, the spin structure $\mathcal{L}=3p$ on $X$ is even and $h^0(X,\mathcal{L})$ is greater of equal than $2$. The dimension $h^0(X,\mathcal{L})$ of the space of holomorphic differentials vanishing to order at least $g-1$ at $p$ is the number of Weierstrass gap numbers $1=\gamma_1<\gamma_2<\dots<\gamma_g$ that are at least $g$. Hence, there are at least $2$ Weierstrass numbers bigger than $4$ if $\mathcal{L}=3p$ is even. In particular, the Weierstrass gap sequence of $(X,p)$ can only be $\mathfrak{G}_p(X)=\{1,2,4,7\}$ and $(X,p)$ is in $\mathcal{M}_{4,1}^{\Gamma}$, where $\Gamma$ be the semigroup generated by $3$ and $5$.\\

 Viceversa, if a pointed Riemann surface $(X,p)$ is in $\mathcal{M}_{4,1}^{\Gamma}$ there is an abelian differential on $X$ vanishing on $p$ with multiplicity $6$. Since the Weierstrass gap sequence of $X$ at $p$ is $\{1,2,4,7\}$, the Riemann surface $X$ cannot be hyperelliptic. By the above argument, the spin structure $\mathcal{L}=3p$ is necessarily even and therefore $(X,p)\in\mathbb{P}\mathcal{H}^{\operatorname{even}}(6)$.
\end{proof}

A description of the pointed Riemann surfaces in the stratum component $\mathbb{P}\mathcal{H}^{\operatorname{even}}(6)$ is available in \cite[Section 4.5]{Chen} and \cite[Section 4.3]{Bullock}. For completeness, we briefly include such a description in the present note. Remarkably, the stratum component $\mathbb{P}\mathcal{H}^{\operatorname{even}}(6)$ is an affine variety. The projective stratum components $\mathbb{P}\mathcal{H}^{\operatorname{odd}}(4)$ and $\mathbb{P}\mathcal{H}(3,1)$ are also affine.

\begin{lemma}
\label{chen}
Let $\Gamma$ be the semigroup generated by $3$ and $5$. Then, the moduli space $\mathcal{M}_{4,1}^{\Gamma}$ is an orbifold of dimension $7$.
\end{lemma}
\begin{proof}[Sketch of the proof]
    Let $(X,p)$ be a non-hyperelliptic pointed Riemann surface of genus $4$ and suppose that $(X,p)\in\mathcal{M}^{\Gamma}_{4,1}$. By Lemma \ref{p3}, we can find an irreducible quadric cone $Q$ and an irreducible cubic $C$ in $\mathbb{P}^3$ such that $X$ is the complete intersection of $Q$ and $C$. Since $7$ is a gap number for $X$ in $p$, the curve $X$ is cut out by a ruling $l_1$ of $Q$ in $p$ with multiplicity $3$. There are also two rulings tangent to $X$ in points $q_1$ and $q_2$ different from the singular point of $Q$.\\

    \begin{figure}[h]
        \centering
        \includegraphics[scale=0.15]{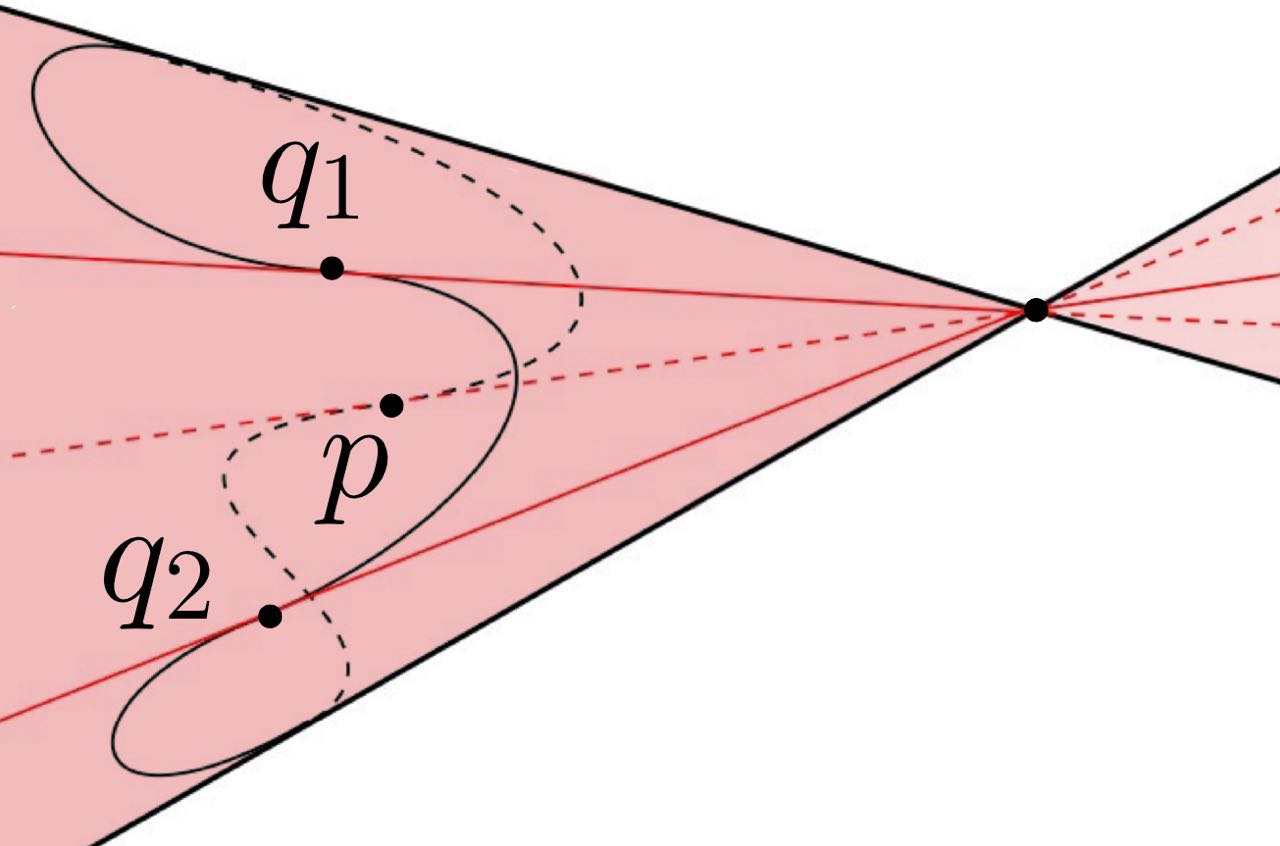}
        \caption{The cubic $C$ on the cone $Q$ with the tangent points $p,q_1,q_2$.}
    \end{figure}

    After a suitable change of coordinates, the tuple $(X,p,q_1,q_2)$ is determined solely by the cubic equation that cuts out $X$ from $Q$. After imposing the tangency requirements of the rulings, we end up with $8$ free non-trivial complex parameters, where any $\lambda\in\mathbb{C}^8\setminus\{0\}$ represents a tuple $(X,p,q_1,q_2)$. However, any two cubic equations define the same isomorphic type of variety up to the action of a matrix in $\operatorname{GL}_4(\mathbb{C})$. The subgroup of $\operatorname{GL}_4(\mathbb{C})$ preserving $Q$ and the three rulings is then isomorphic to $\mathbb{C}^*$. The locus in $\mathbb{C}^8\setminus\{0\}$ parametrizing singular curves is a hypersurface and the moduli space $\mathcal{M}^{\Gamma}_{4,1}$ is covered by its complement in $\mathbb{C}^8\setminus\{0\}$ by the action of $\mathbb{C}^*$.
\end{proof}

The following is a consequence of Proposition \ref{iso}, Proposition \ref{gapstrata} and Lemma \ref{chen}.

\begin{corollary}
\label{isofinal}
Let $\Gamma$ be the semigroup generated by $3$ and $5$. The orbifolds $\mathbb{P}\mathcal{H}^{\operatorname{even}}(6)$ and $\mathcal{M}_{4,1}^{\Gamma}$ are isomorphic. 
\end{corollary}

\section{Versal deformation spaces of plane curve singularities}

In this section, we revise some results on the moduli spaces $\mathcal{M}_{g,1}^{\Gamma}$ for $\Gamma$ semigroup in $\mathbb{N}$ due to Pinkham \cite{Pinkham}. In general, it is even hard to establish whether these moduli spaces are empty or not. On the other hand, we do have some results in low genera; see, for example, \cite{Nakano}.\\

Suppose $\Gamma$ is the semigroup in $\mathbb{N}$ with $\{a_1,\dots,a_k\}$ as a minimal generating set and consider the monomial curve $C_{\Gamma}=\{(t^{a_1},\dots,t^{a_k})\in\mathbb{C}^k\mid t\in\mathbb{C}\}$. Every monomial curve $C_\Gamma$ has an isolated singularity at the origin and the $1$-dimensional algebraic torus $\mathbb{C}^*$ acts naturally on the parameter $t\in\mathbb{C}$ of $C_\Gamma$. Pinkham proved that the moduli space $\mathcal{M}_{g,1}^{\Gamma}$ is a quotient of the \textit{versal deformation spaces} of the monomial curve $C^\Gamma$ \cite[Proposition 13.9]{Pinkham}. In what follows, we will recall the definition of the versal deformation of the monomial curve $C_{\Gamma}$ in the particular case $C_{\Gamma}$ is the zero level set of a germ of the complex analytic map $f:\mathbb{C}^2\rightarrow\mathbb{C}$ with an isolated singularity at the origin. In particular, we will focus on the case $f$ arises from an irreducible root system $R$. For more details see, for example, \cite[Section 2]{Cuadrado2021} or \cite[Chapter II, Section 1.3]{Greuel07}.\\

A versal deformation of $C_{\Gamma}$ is a morphism of complex analytic varieties. Roughly speaking, the preimage of a fixed base point in the target is isomorphic to the monomial curve $C_{\Gamma}$, while the dimensions of the fibers are locally preserved. More precisely, consider the algebra $\mathbb{C}\{x,y\}$ of convergent power series in two complex variables. A classical result states that the algebra $\mathbb{C}\{x,y\}$ quotient by the ideal $(f_x,f_y)$ generated by the partial derivatives of $f$ has a finite dimension if $f$ is a plane curve with an isolated singularity in the origin. Suppose now that $C_{\Gamma}$ is defined by the germ $f_\Gamma$. Then, there are polynomials $g_1,\dots,g_m\in\mathbb{C}[x,y]$ projecting to generators of $\mathbb{C}\{x,y\}/(f_x, f_y)$. Consider the perturbation of $f_\Gamma$ $$F_\Gamma(x,y,s)=f_\Gamma(x,y)+\sum_{i=1}^m s_ig_i(x,y)$$ given by the parameters $s=(s_1,\dots,s_m)\in\mathbb{C}^m$ and the monomials $g_1,\dots,g_m\in\mathbb{C}[x,y]$. In the affine coordinates $(x,y,s_1,\dots,s_m)$ on $\mathbb{C}^2\times\mathbb{C}^m$, the projection $\mathbb{C}^2\times\mathbb{C}^m\rightarrow\mathbb{C}^m$ can be restricted to the vanishing locus of the polynomial $F_\Gamma$. The map $\pi_\Gamma:\mathbb{V}(F_\Gamma)\rightarrow\mathbb{C}^m$ is the versal deformation of $C_\Gamma$.  Note that the fiber at the origin coincides with $C_\Gamma$.\\

The set $U_\Gamma$ of $s\in\mathbb{C}^m$ such that the fiber $\pi_\Gamma^{-1}(s)$ is smooth is the \textit{versal deformation space} of $C_\Gamma$.
The smooth fibers of $\pi_\Gamma$ can projectivized each fiber by adding a point at infinity. More precisely, we can homogenize the polynomials $F_\Gamma(s,\cdot,\cdot)$ in the variables $(x,y)$ for every $s\in\mathbb{C}^{m}$ and denote the associated projective variety by $\mathbb{V}(F_{\Gamma,s})$ The following is Pinkham's result. 

\begin{theorem}
\label{pinkham}
     If $U_\Gamma$ is not empty, the $\mathbb{C}^*$ action on $C_\Gamma$ can be extended to $U_\Gamma$, in such a way that $\pi_\Gamma$ is $\mathbb{C}^*$-equivariant and $U_\Gamma/\mathbb{C}^*$ is isomorphic to $M^\Gamma_{g,1}$. The isomorphism is given by $U_\Gamma\ni s\mapsto\mathbb{V}(F_{\Gamma,s})\in M^\Gamma_{g,1}$, where the marked point of the Riemann surface $\mathbb{V}(F_{\Gamma,s})$ is the added point at infinity.
\end{theorem}

Let $R$ be one of the irreducible root systems of type $A_n,D_n,E_6,E_7$ or $E_8$ for $n\in\mathbb{N}_{\geq 3}$. Each root system comes with a germ of a complex analytic map $f_R$, as in the table below. Suppose that, up to a change of coordinates\footnote{Up to a change of coordinate in $\mathbb{C}^2$, versal deformation spaces are homeomorphic.}, the monomial curve $C_\Gamma$ is defined by  $f_R$ for some root system $R$. For simplicity, we will denote the versal deformation space of $f_R$ by $U_R$.

\vspace{5mm}

\begin{center}
\begin{tabular}{|c|c|}
\hline
     Root system $R$  & Germ $f_R$ \\
     \hline
$A_n$ & $x^2+y^{n+2}$\\ 
\hline
$D_n$ & $y(x^2+y^{n-2})$\\ 
\hline
$E_6$ & $x^3+y^4$\\
\hline
$E_7$ & $x(x^2+y^3)$\\
\hline
$E_8$ & $x^3+y^5$\\
\hline
\end{tabular}
\end{center}

\vspace{5mm}

The following is a Theorem of Arnol'd \cite[Propositions 9.1-9.3]{arnold}. 

\begin{theorem}
\label{arnold}
    Let $R$ be one of the irreducible root systems of type $A_n,D_n,E_6,E_7$ or $E_8$ for $n\in\mathbb{N}_{\geq 3}$. Consider the complement $V_R$ of the real hyperplane arrangement associated with $R$. The versal deformation $U_R$ is homeomorphic to the complexification of $V_R$ modulo the action of the Coxeter reflection group $W_R$. In particular, the versal deformation space $U_R$ is an Eilenberg-MacLane space $K(\pi,1)$ for the Artin group $A_R$.
\end{theorem}

\textit{Artin groups} are finitely presented groups. Given a finite tree\footnote{Artin group are generally defined from labelled graphs. Here, we will only consider \textit{small-type} Artin groups, and we will not need any labels on the defining graphs.} $\Gamma$, Artin groups have generators defined from the set of vertices $\mathcal{V}(\Gamma)=\{v_1,\dots,v_n\}$, while the relations come from the edges, as follows
\begin{align*}
    A_\Gamma = \Bigg\langle a_1, \dots, a_n \in \mathcal{V}(\Gamma) \text{ }\Bigg\vert
    \text{ }
    \begin{array}{@{}l@{}l@{}l@{}}
         a_i a_j a_i &= a_j a_i a_j &\quad \text{if } v_i \text{ and } v_j \text{ are adjacent}\\
         a_i a_j &= a_j a_i &\quad \text{otherwise}
    \end{array}
    \Bigg\rangle. 
\end{align*} 

Any Artin group surjects in a Coxeter group by imposing every standard generator to be an involution. If the associated Coxete group is finite, then the Artin group is called of \textit{finite-type}, and in these cases, the complexified $V_R$ modulo the action of $W_R$ is $K(\pi,1)$ by a theorem of Deligne \cite{Deligne1972}. Every Artin group $A_R$ of Theorem \ref{arnold} is finite-type.\\

The isomorphism of Theorem \ref{arnold} is given by a basis of polynomials generating the algebra of polynomials that are invariant under the action of $W_R$. If $R=E_8$, the basis of the algebra consists of homogeneous polynomials $f_1,\dots,f_8\in\mathbb{C}[x_1,\dots,x_8]$ of even degree. In particular, the map \begin{align*}
    \tau_R:\mathbb{C}^8&\rightarrow\mathbb{C}^8\\
    x&\mapsto(f_1(x),\dots,f_8(x))
\end{align*} induces an homeomorphism between $\mathbb{C}^8/W_{E_8}$ and $\mathbb{C}^8$ such that the complexification of $V_R$ modulo the action of Coxeter group $W_R$ is mapped homeomorphically to $U_R$.\\

In case $\Gamma$ is generated by $\{3,5\}$, the monomial curve $C_\Gamma$ is, up to change of coordinates, the vanishing locus of $f_R=x^3+y^5$, where $R$ is the root system $E_8$. The following is a consequence of Theorem \ref{arnold} and Corollary \ref{isofinal}.

\begin{theorem}
    The stratum component $\mathcal{H}^{\operatorname{even}}(6)$ is an $K(\pi,1)$ orbifold classifying space.
\end{theorem}
\begin{proof}
A good orbifold is $K(\pi,1)$ if covered by a contractible manifold and $\mathcal{H}^{\operatorname{even}}(6)$ is $K(\pi,1)$ if $\mathbb{P}\mathcal{H}^{\operatorname{even}}(6)$ is. The projective stratum component $\mathbb{P}\mathcal{H}^{\operatorname{even}}(6)$ is covered by the versal deformation space $U_R$ for $R=E_8$, that is $K(\pi,1)$ manifold and therefore covered by a contractible manifold.
\end{proof}

\section{The orbifold fundamental group}

In this section, we show that the orbifold fundamental group of $\mathbb{P}\mathcal{H}^{\operatorname{even}}(6)$ is isomorphic to the inner automorphism group of the Artin group associated with the $E_8$ root system. In particular, the kernel of the monodromy is very large and contains a non-abelian free group of rank $2$. Hence, the connected components of the Teichm\"uller cover of the stratum component $\mathcal{H}^{\operatorname{even}}(6)$ have a non-trivial fundamental group.\\

Let $G$ be a topological group acting properly on a manifold $X$ and let $EG\rightarrow BG$ be the universal $G$-bundle. The \textit{Borel construction} $X_G$ is the quotient of $EG\times X$ by the diagonal action of $G$ on both factors. The orbifold fundamental group of $X/G$ is isomorphic to the fundamental group $\pi_1(X_G)$ and $X_G\rightarrow BG$ is a fiber bundle with $X$ as a fiber \cite[Chapter 2, Theorem 2.18]{Orbifolds}. Since the projective stratum component $\mathbb{P}\mathcal{H}^{\operatorname{even}}(6)$ is the quotient of the versal deformation space $U_R$, for the root system $R=E_8$, by $\mathbb{C}^*$ we get the short exact sequence
\begin{align}
\label{quot}
    1\rightarrow\pi_1(\mathbb{C}^*)\rightarrow\pi_1(U_R)\rightarrow\pi_1^{orb}(\mathbb{P}\mathcal{H}^{\operatorname{even}}(6))\rightarrow1
\end{align} 
from the fiber bundle associated with Borel construction.\\

By Theorem \ref{arnold}, the fundamental group of $U_R$ is the Artin group $A_{E_8}$ and its quotient by a cyclic normal subgroup is isomorphic to $\pi_1^{orb}(\mathbb{P}\mathcal{H}^{\operatorname{even}}(6))$. 

\begin{lemma}
\label{central}
    Every cyclic normal subgroup of $A_{E_8}$ is central. 
\end{lemma}
\begin{proof}
     Let $a\in A_{E_8}$ be the generator of an infinite cyclic normal subgroup. For every $g\in A_{E_8}$ there is an $n\in\mathbb{N}$ such that $gag^{-1}=a^n$ holds. Standard generators of $A_{E_8}$ share only length-preserving relations. Therefore, there exists a well-defined homomorphism $$\operatorname{deg}:A_{E_8}\rightarrow\mathbb{Z}$$ assigning the standard generators length $1$. The following inequality shows that $n$ must be equal to $1$, provided $g\neq id$: $$\operatorname{deg}(a)=\operatorname{deg}(gag^{-1})=\operatorname{deg}(a^n)=n\operatorname{deg}(a).$$ Therefore, the normal subgroup $\langle a\rangle$ is central.
\end{proof}

In particular, the orbifold fundamental group of $\mathbb{P}\mathcal{H}^{\operatorname{even}}(6)$ is isomorphic to the quotient of $A_{E_8}$ by a central cyclic subgroup. The center of $A_{E_8}$ is infinite cyclic and generated by the \textit{Garside element} $\Delta_{E_8}$. We will show that $\Delta_{E_8}$ generates the central subgroup of Lemma \ref{central} and in particular that the orbifold fundamental group $\pi_1^{orb}(\mathbb{P}\mathcal{H}^{\operatorname{even}}(6))$ is isomorphic to the inner automorphism group $\operatorname{Inn}(A_{E_8})$.\\

Let $R$ be the root system $E_8$ and denote by $V_R$ the open complement in $\mathbb{R}^8$ of the hyperplanes family $\{H_\alpha\mid\alpha\in I_R\}$ associated to $R$. The Artin group $A_{E_8}$ has an interpretation as a fundamental group by Theorem \ref{arnold}. Let us pick a chamber $C\subset V_R$ and a point $p\in C$. The fundamental group of the complexification of $V_R$, denoted by $\mathbb{V}_R$, modulo the Coxeter group $W_R$ and based at the point represented by $p$ is isomorphic to $A_{E_8}$.\\

We now construct the Garside element $\Delta_{R}$ as the homotopy class of a loop in $\mathbb{V}_R/W_R$. The following construction is due to Brieskorn \cite{Brieskorn} and can also be found in \cite[Section 2]{Looijenga2008}. For every $x\in V_R$, we define $C_x$ to be the either $V_R$, if $x$ is not contained in any hyperplane $H_\alpha$, or the intersection of all open half-spaces $H_\alpha^+$ containing the chamber $C$ and bounded by $H_\alpha$ if $x\in H_\alpha$. The set $$\mathbb{U}=\{x+iy\mid y\in C_x\}$$ is an open subset of $\mathbb{V}_R$ and it is star-like with respect to any point in $iC$. Therefore, the set $\mathbb{U}$ is contractible. As a result, there is a unique homotopy type of arc $\gamma_R$ between $p$ and $-p$ entirely contained in $\mathbb{U}$. Since $-id_{V_R}\in W_R$ in case $R=E_8$, the arc $\gamma_R$ projects to a loop in $\mathbb{V}_R/W_R$. The Garside element $\Delta_{E_8}$ can be interpreted as the homotopy class of $[\gamma_R]$ in $\mathbb{V}_R/W_R$.\\

The arc $\gamma_R$ can be taken to be the composition $\delta*\sigma$ of the following path segments
\begin{align*}
    \sigma:[0,1]&\rightarrow\mathbb{U}&\delta:[0,1]&\rightarrow\mathbb{U}\\
    t&\mapsto h(t)p & t&\mapsto ih(t)p,
\end{align*}

where $h(t)=(1-t)+it$.

\begin{figure}[h]
    \centering
    \begin{minipage}{0.45\textwidth}
        \centering
        \includegraphics[width=0.8\textwidth]{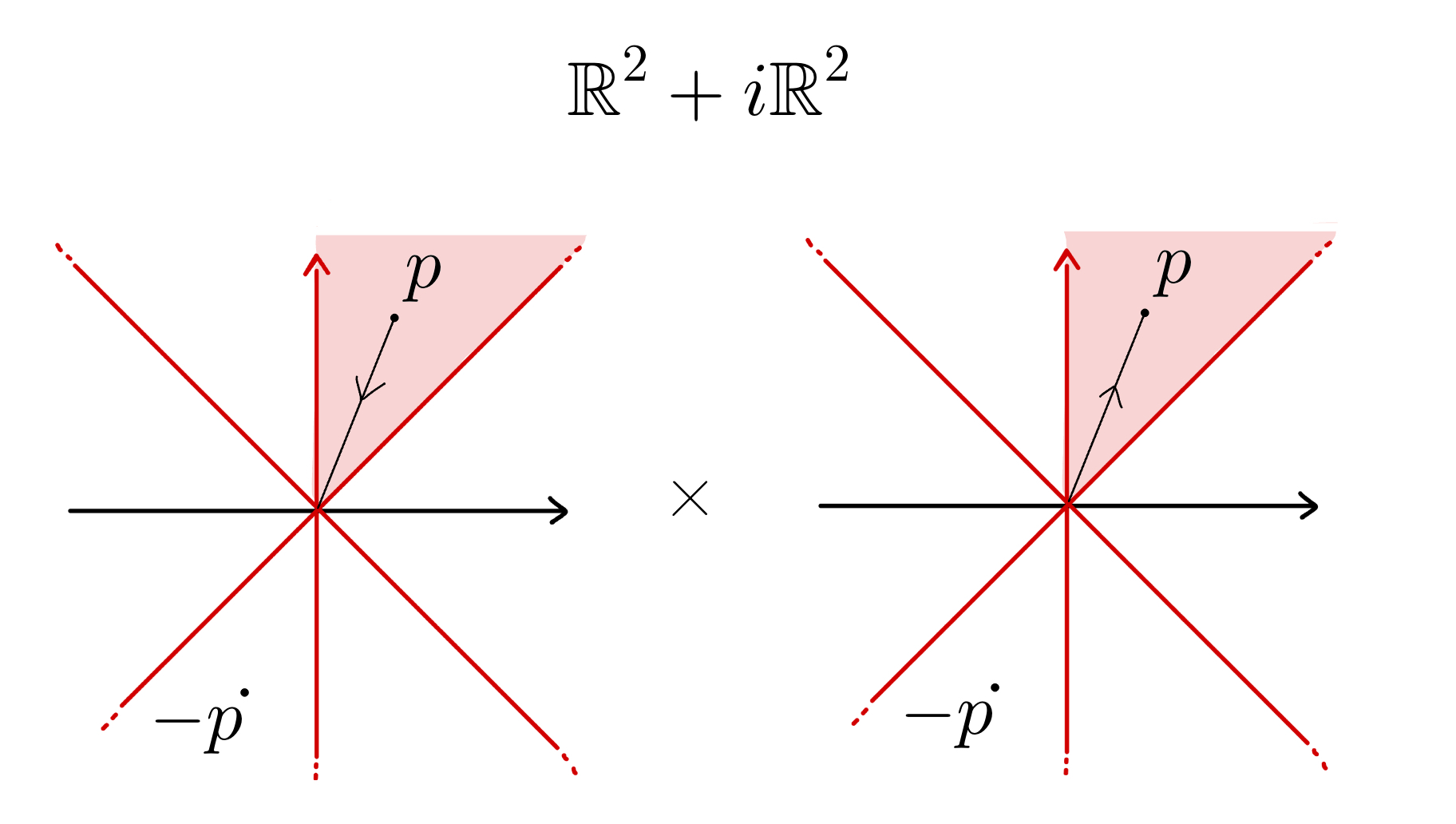} 
    \end{minipage}\hfill
    \begin{minipage}{0.45\textwidth}
        \centering
        \includegraphics[width=0.8\textwidth]{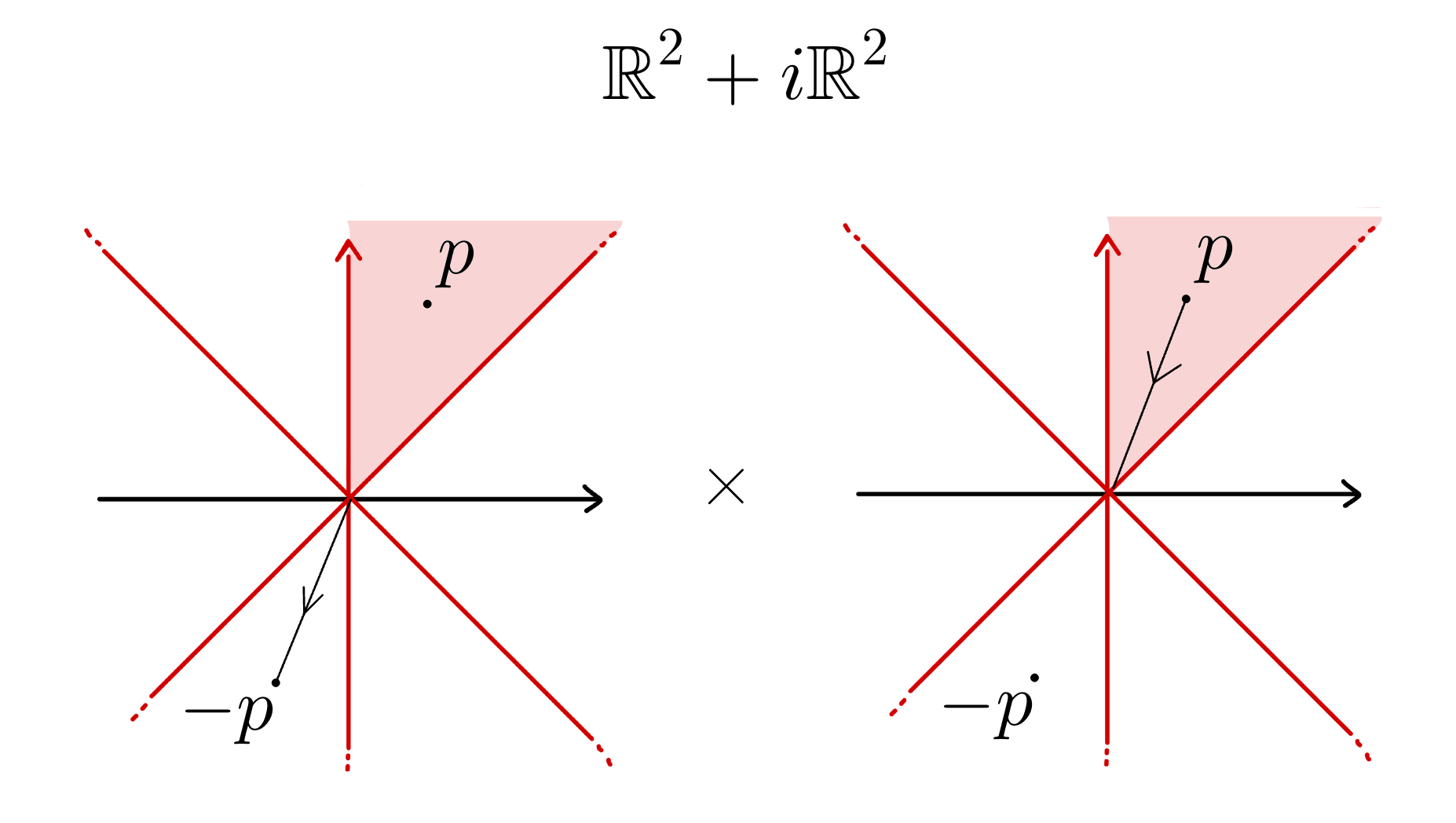} 
    \end{minipage}
    \caption{An example of the path segments $\sigma$ (on the left of the picture) and of $\delta$ (on the right hand side of the picture) in the case $R=A_3$. The coloured area represents the complexified chamber $C+iC$.}
\end{figure}

\begin{proposition}
    Let $R$ be the root system $E_8$. The image by $\tau_R$ of the homotopy class of the loop $[\gamma_R]$ in $\mathbb{V}_R/W_R$ generates the fundamental group of the $\mathbb{C}^*$-fiber associated to the quotient map $U_R\rightarrow\mathbb{P}\mathcal{H}^{\operatorname{even}}(6)$.
\end{proposition}
\begin{proof}
    We want to show that the homotopy class of the loop $\tau_{R}(\gamma_R)$ generates the fundamental group of the $\mathbb{C}^*$-fiber associated with the quotient map $U_R\rightarrow\mathbb{P}\mathcal{H}^{\operatorname{even}}(6)$. The punctured complex plane $\mathbb{C}^*$ acts on $U_{E_8}$ component-wise with weights given by the degrees $d_1,\dots,d_8$ of the homogeneous polynomials $f_1,\dots,f_8$. In particular, the great common divisor of $d_1,\dots,d_8$ is $2$ and the the fiber $\mathcal{O}_p=\{(\lambda^{d_1}p_1,\dots,\lambda^{d_8}p_8)\mid\lambda\in\mathbb{C}^*\}$ of $p\in U_R$ is homeomorphic to $\mathbb{C}^*/\mathbb{Z}_2$ where the underlying relation is given by $z\sim -z$. The fundamental group of $\mathcal{O}_p$ is isomorphic to $\mathbb{Z}$ and generated by the image of any arc in $\mathbb{C}^*$ tracing an angle of $\pi$. The arc $\gamma_R$ traces an angle of $\pi$ between the endpoints $p$ and $-p$ and therefore the image 
    \begin{align*}
        \tau_{E_8}(\gamma_R):[0,1]&\rightarrow\mathbb{U}\\
    t&\mapsto\biggl\{ \begin{array}{cc}
       (h(2t)^{d_1}f_1(p),\dots,h(2t)^{d_8}f_8(p))  & \text{if }t\in[0,\frac{1}{2}] \\
       (i^{d_1}h(2t-1)^{d_1}f_1(p),\dots,i^{d_1}h(2t-1)^{d_8}f_8(p))  & \text{if }t\in[\frac{1}{2},1].
    \end{array}
    \end{align*}

    represents a generator of the fundamental group of the $\mathbb{C}^*$-fiber $\mathcal{O}_p$.
\end{proof}

We obtain the following result from the short exact sequence in (\ref{quot}).

\begin{corollary}
\label{fundgrp}
    The orbifold fundamental group of $\mathbb{P}\mathcal{H}^{\operatorname{even}}(6)$ is isomorphic to the inner automorphism group $\operatorname{Inn}(A_{E_8})$. Then, the group $\pi_1^{orb}(\mathcal{H}^{\operatorname{even}}(6))$ is a central extension of $\operatorname{Inn}(A_{E_8})$.
\end{corollary}

\section{The kernel of the monodromy map}

In this section we prove the following.

\begin{theorem}[Large Kernel Property]
\label{lkp}
    The kernel of the monodromy $\rho:\pi_1^{orb}(\mathcal{H}^{\operatorname{even}}(6))\rightarrow\operatorname{Mod}(\Sigma_{4,1})$ contains a non-abelian free group of rank $2$.
\end{theorem}

Corollary \ref{fundgrp} implies that the monodromy $\rho:\pi_1^{orb}(\mathcal{H}^{\operatorname{even}}(6))\rightarrow\operatorname{Mod}(\Sigma_{4,1})$ factors through a homomorphism from $\operatorname{Inn}(A_{E_8})$ to the mapping class group $\operatorname{Mod}(\Sigma_{4,1})$ of a genus $4$ closed surface with a marked point.  Indeed, the following diagram commutes

\[
\xymatrix{
  & \pi_1^{orb}(\mathcal{H}^{\operatorname{even}}(6)) \ar[d]^{} \ar[dl]_-{} \\
   \pi_1^{orb}(\mathbb{P}\mathcal{H}^{\operatorname{even}}(6)) \ar[r]^{} & \operatorname{Mod}(\Sigma_{4,1}),
}
\]

where $\rho$ is the vertical map and the oblique one is induced by the principal $\mathbb{C}^*$-bundle $q_\mathcal{C}:\mathcal{TC}\rightarrow\mathbb{P}(\mathcal{TC})$ for $\mathcal{C}=\mathcal{H}^{\operatorname{even}}(6)$.
The horizontal map is the monodromy $\rho_\mathbb{P}:\pi_1^{orb}(\mathbb{P}\mathcal{H}^{\operatorname{even}}(6))\rightarrow\operatorname{Mod}(\Sigma_{4,1})$ associated to the projective stratum component $\mathbb{P}\mathcal{H}^{\operatorname{even}}(6)$. Notice that, if the Large Kernel Property holds for the monodromy $\rho_\mathbb{P}$, then Theorem \ref{lkp} follows immediately.\\

The monodromy $\rho_{\mathbb{P}}$ is a homomorphism from $\operatorname{Inn}(A_{E_8})$ to $\operatorname{Mod}(\Sigma_{4,1})$, induced by a \textit{geometric homomorphism}. Geometric homomorphisms are maps form Artin groups $A_\Gamma$ to mapping class groups such that standard generators are mapped to Dehn twists about curves with an intersection pattern that respects the pattern of $\Gamma$. The following is a classical theorem in the theory of plane curve singularities. See, for example, \cite[Chapter 3]{arnoldbook}.

\begin{theorem}[Picard-Lefschetz Theorem] Let $R$ be one of the irreducible root systems of type $A_n,D_n,E_6$, $E_7$ or $E_8$. The monodromy $\pi_1(U_R)\rightarrow\operatorname{Mod}(\Sigma_{g,n})$ of the versal deformation space is a geometric homomorphism.
\end{theorem}

\begin{figure}[h]
    \centering
    \includegraphics[scale=0.2]{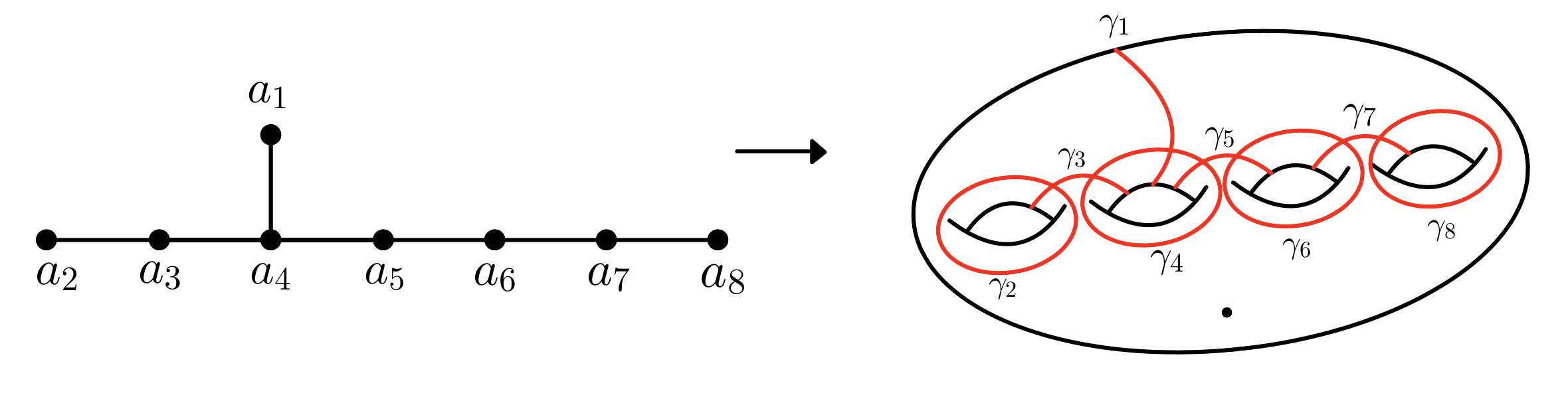}
    \caption{A correspondence between the $E_8$ Dinkin diagram on some closed curves on $\Sigma_{4,1}$. Each vertex corresponds to a simple closed curve on the punctured surface on the right-hand side. The geometric homomorphism sends each standard generator of $A_{E_8}$ to the corresponding Dehn twist.}
    \label{fig:enter-label}
\end{figure}

The theorem below has been proved in \cite[Theorem B]{Giannini2023} and builds on the work of Wajnryb \cite{Wajnryb1999} using the acylindrical hyperbolicity of finite-type Artin groups verified by Calvez-Wiest \cite{CalvezWiestAcyArt2017}, a generalized notion of hyperbolicity. Here, a Ping-Pong strategy detects non-abelian free groups.

\begin{theorem}
\label{artin}
    Suppose a finite graph $\Gamma$ contains the Dynkin diagram $E_6$ as a subgraph. The kernel of any geometric homomorphism of $A_\Gamma$ contains a copy of the non-abelian free group of rank $2$.
\end{theorem}

\begin{proof}[Proof of Theorem \ref{lkp}]
    The Dynking diagram $E_8$ contains $E_6$ as a subgraph and by Thereom \ref{artin} any geometric homomorphism $A_{E_8}\rightarrow\operatorname{Mod}(\Sigma_{4,1})$ has the Large Kernel Property. The versal deformation space $U_{E_8}$ comes with a monodromy $\pi_1(U_{E_8})\rightarrow\operatorname{Mod}(\Sigma_{4,1})$ that is a geometric homomorphism by the Picard-Lefschetz Theorem. However, the monodromy $\rho_\mathbb{P}$ can be obtained from the monodromy of $U_{E_8}$ by taking the quotient of the domain by the Garside element $\Delta_{E_8}$. The copy of the non-abelian free group of rank $2$ of Theorem \ref{artin} embedds in $\operatorname{Inn}(A_{E_8})$. Consequently, the claim is proved for $\rho_{\mathbb{P}}$ as $\ker\rho_{\mathbb{P}}$ is large.
\end{proof}

\end{document}